\documentclass[12pt,a4paper,fleqn,reqno]{amsart}
\pdfoutput=1

\usepackage[utf8]{inputenc}
\usepackage[T1]{fontenc}
\usepackage[british]{babel}
\usepackage[selected=true,shrink=15,stretch=25,letterspace=30]{microtype}

\usepackage{amsaddr}

\usepackage{setspace}
\usepackage{paralist}
\usepackage{fourier}
\usepackage[margin=20mm,bottom=22mm,top=18mm]{geometry}
\usepackage[pdfborderstyle={/S/U/W .5}]{hyperref}

\newcommand{\mh}{\ensuremath{\mathfrak h}}
\newcommand{\ms}{\ensuremath{\mathfrak s}}
\newcommand{\mt}{\ensuremath{\mathfrak t}}
\newcommand{\Detm}{\textup{Det}}
\newcommand{\IDetm}{\textup{Indet}}
\newcommand{\ind}{\operatorname{ind}}
\newcommand{\supp}{\operatorname{supp}}
\newcommand{\coloneqq}{\mathrel{\mathop:}=}
\newcommand{\eqqcolon}{=\mathrel{\mathop:}}

\newtheorem*{theorem*}{Theorem}
\newtheorem{theorem}{Theorem}
\newtheorem{lemma}[theorem]{Lemma}
\newtheorem{corollary}[theorem]{Corollary}

\theoremstyle{remark}
\newtheorem{remark}[theorem]{Remark}

\begin{document}
\author{Alexander Dyachenko}
\email{\href{mailto:diachenko@sfedu.ru}{diachenko@sfedu.ru}}
\email{\href{mailto:dyachenk@math.tu-berlin.de}{dyachenk@math.tu-berlin.de}}
\address{
    TU-Berlin, Institut f\"ur Mathematik, Sekr.~MA 4-2,\\
    Straße des 17. Juni 136, 10623 Berlin, Germany.
}

\title%[Rigidity of moment sequences]%
{Rigidity of the Hamburger and Stieltjes moment sequences}

\date{1st December 2017}

\keywords{Moment problems $\cdot$ Index of determinacy}
\subjclass[2010]{30E05 $\cdot$ 44A60 $\cdot$ 42C05}

\begin{abstract}
    This paper aims at finding conditions on a Hamburger or Stieltjes moment sequence, under
    which the change of at most a finite number of its entries produces another sequence of the
    same type. It turns out that a moment sequence allows all small enough variations of this
    kind precisely when it is indeterminate. We also show that a determinate moment sequence has
    the finite index of determinacy if and only if the corresponding finite number of its
    entries can be changed in a certain way.
\end{abstract}

\maketitle

\section{Brief introduction}
The classical moment problems by Stieltjes and Hamburger play an important role in many branches
of mathematics. They consist in finding a distribution of masses (a positive measure) based on a
sequence of the real numbers which are called moments. Sequences of the moments can be
characterized by positivity of the related Hankel quadratic forms, and the induced
interrelationship between the entries is relatively strong. Indeed, although increasing the
leading (i.e.~zeroth) moment is always possible, other changes of finitely many moments turn to
be impossible for many moment sequences: here we call such sequences ``rigid''. At the same
time, ``nonrigid'' moment sequences may allow more or less free variations of their entries, and
the criterion for this freedom seems to be absent in literature.

We use the so-called \emph{index of determinacy} to express the tightness of the conditions
arising from positivity of the Hankel forms. Our first goal is to describe its connection to the
rigidity of moment sequences. In particular, we show that this index determines the minimal
number of leading moments which can be varied.

The second goal is to find out whether or not these variations can be arbitrary. It turns out
that indeterminate moment sequences survive all small enough changes of finitely many entries.
At the same time, determinate sequences may survive only specific variations:
\emph{one of the moments allows all small changes if and only if the corresponding moment
    problem is indeterminate.}

The author is very grateful to Alan Sokal and Christian Berg for their remarks, as well as for
introducing the problem and the initial idea. This research was supported by the Einstein
Foundation Berlin.

\section{Definitions and basic facts}

For introducing rigorous statements and proofs, we need certain basic facts and definitions,
most of which can be found in the classical books~\cite{Akhiezer,ShohatTamarkin}.

Let~$\mh = (c_i)_{i=0}^\infty=(c_0,c_1,c_2,\dots)$ be a sequence of the Hamburger moments, that
is let there exist some (positive) measure%
\footnote{We introduce a slight abuse in the notation by using the differential~$d\mu(x)$ of a
    non-decreasing function~$\mu(x)$ for denoting the related measure.}%
~$d\mu(x)$ on the real line such that
\begin{equation}\label{eq:1}
    c_i = \int_{-\infty}^\infty x^i \, d\mu(x) \quad
    \text{for all}\quad
    i=0,1,\dots.
\end{equation}
The work~\cite{HamburgerT1} established that~\mh\ is a Hamburger moment sequence if and only if
all quadratic forms
\begin{equation}\label{eq:1qf}
    \sum_{i,j=0}^p c_{i+j}x_ix_j,
    \quad p=0,1,2,\dots
\end{equation}
are positive semidefinite. Given a Hamburger moment sequence~\mh, we write~$\mh\in\Detm_H$ if it
is determinate, i.e. if the measure~$d\mu(x)$ is uniquely determined by~\eqref{eq:1},
or~$\mh\in\IDetm_H$ otherwise.

Note that the trimmed sequence of moments~$(c_i)_{i=2n}^\infty$ corresponds for
each~$n=1,2,\dots$ to the measure~$x^{2n}d\mu(x)$. If~$(c_i)_{i=2n}^\infty$ determines the
unique measure~$x^{2n}d\mu(x)$, then~$d\mu(x)$ is uniquely determined outside the origin, while
the atom at the origin is fixed by the moment~$c_0$. That is, prefixing a moment sequence with a
pair of new entries only means introducing additional constraints, which cannot make the problem
``less definite'':
\begin{lemma}\label{lm:1}
    If~$\mh = (c_i)_{i=0}^\infty\in\IDetm_H$, then~$(c_i)_{i=2n}^\infty\in\IDetm_H$
    for~$n=1,2,\dots$. Accordingly, if~$(c_i)_{i=2n}^\infty\in\Detm_H$ for some positive
    integer~$n$, then~$\mh\in\Detm_H$.
\end{lemma}
This lemma allows introducing the \emph{index of determinacy} (see~\cite{BergDuran})
\[
    \ind_0(\mh) \coloneqq\sup\big\{n\in\mathbb{Z}_{\ge0}:
      (c_i)_{i=2n}^\infty\in\Detm_H% , \quad \text{where~$c_i$ are given in~\eqref{eq:1}}
      \big\}
\]
of a determinate Hamburger sequence~\mh. In other words,~$\ind_0(\mh)$ is the minimal
nonnegative number such that the measure~$x^{2n}d\mu(x)$ with an integer~$n$ is uniquely
determined by its moments as soon as~$0\le n\le\ind_0(\mh)$ and not uniquely determined for
all~$n>\ind_0(\mh)$. The work~\cite{BergDuran} introduces the index of determinacy through the
measure, which allows calculating it with respect to other points than the origin. Both
definitions coincide for the origin, and ours better fits to the current study. It turns out
that, if~$\ind_0\mh<\infty$, then the corresponding measure~$\mu(x)$ is discrete. Another result
of~\cite{BergDuran} is that the indices for distinct points can differ by at most~$1$.

Each distribution of masses~$d\mu(z)$ satisfying~\eqref{eq:1} determines a unique%
\footnote{The uniqueness follows from the Stieltjes-Perron inversion
    formula~\cite[pp.~123--126]{Akhiezer}} %
mapping of the upper half-plane~$\{z\in\mathbb C:\Im z >0\}$ into itself by the
formula~\cite[p.~95, Thm.~3.2.1]{Akhiezer}
\[
\int_{-\infty}^\infty \frac{d\mu(x)}{x-z},
\]
where the integral has the asymptotic expansion
\begin{equation}\label{eq:asymptotic}
    -\frac{c_{0}}{z^{1}}-\frac{c_{1}}{z^{2}}-\frac{c_{2}}{z^{3}}-\cdots
    \quad\text{as}\quad z\to +\infty\cdot i.
\end{equation}
This fact yields a characterization of determinate Hamburger moment sequences: the related
asymptotic series~\eqref{eq:asymptotic} represents only one mapping of the upper half-plane into
itself.

\begin{remark}\label{rem:defin-basic-facts}
    As follows from Kronecker's studies, singularity of any of the positive semidefinite
    forms~\eqref{eq:1qf} is equivalent to that the support of the corresponding
    measure~$d\mu(x)$ has finitely many points. Supports of such measures are compact, and thus
    the related index of determinacy is always infinite: the series~\eqref{eq:asymptotic} turn
    to the Taylor series at infinity. Theorem~\ref{th:2} and Corollary~\ref{cr:1} in
    Section~\ref{sec:rigid-nonr-moment} show that the moment sequences of such measures are
    rigid. Sections~\ref{sec:results-strong-vers}--\ref{sec:proofs2} only deal with nonrigid
    sequences, so the corresponding forms~\eqref{eq:1qf} are necessarily positive definite
    therein.
\end{remark}

The sequence~$\ms = (s_i)_{i=0}^\infty$ is called the Stieltjes moment sequence provided that
there exists some measure~$d\sigma(x)$ on the positive semi-axis such that
\begin{equation}\label{eq:1s}
    s_i = \int_{0}^\infty x^i \, d\sigma(x) \quad
    \text{for all}\quad
    i=0,1,\dots.
\end{equation}
The set of Stieltjes moment sequences is hence embedded in the set of Hamburger moment
sequences. Moreover, the sequence~\ms\ is a Stieltjes moment sequence if and only if both
sequences~$\ms$ and~$(s_i)_{i=1}^\infty$ are Hamburger moment sequences. We
write~$\ms\in\Detm_S$ if~$\ms$ is determinate, i.e. if the measure~$d\sigma(x)$ is uniquely
determined by~\eqref{eq:1s}, or~$\ms\in\IDetm_S$ otherwise. Although~$\IDetm_S$ contains only
those Stieltjes sequences that belong to~$\IDetm_H$, Hamburger proved~\cite{HamburgerT2}
that~$\Detm_S\cap\IDetm_H$ is non-empty and gave a description of this set. Moreover, the
inclusion~$\ms\in\Detm_S\cap\IDetm_H$ implies that the corresponding unique Stieltjes measure
has an atom at the origin, see~\cite{BergThill}. By analogy with Lemma~\ref{lm:1} we have:
\begin{lemma}\label{lm:1s}
    If~$\ms\in\IDetm_S$, then~$(s_i)_{i=n}^\infty\in\IDetm_S\subset\IDetm_H$ for~$n=1,2,\dots$.
    If~$(s_i)_{i=n}^\infty\in\Detm_S$ for some positive integer~$n$, then~$\ms\in\Detm_S$.
\end{lemma}

For determinate Stieltjes moment sequences, the \emph{index of determinacy} is defined
(see~\cite{BergThill}) by
\[
    \ind(\ms) \coloneqq\sup\big\{n\in\mathbb{Z}_{\ge0}:
    (s_i)_{i=n}^\infty\in\Detm_S\big\}
    .
\]
Accordingly, the measure~$x^{n} d\sigma(x)$ is the only one corresponding
to~$(s_i)_{i=n}^\infty$ when~$n\le\ind(\ms)$, and there are infinitely many measures corresponding
to$~(s_i)_{i=n}^\infty$ when~$n>\ind(\ms)$.

The Hamburger moment sequence~\mh\ is called symmetric if all odd moments~$c_{2n+1}$,
~$n\in\mathbb{Z}_{\ge0}$, are equal to zero. The corresponding measure~$d\mu(x)$ is not
necessarily symmetric (i.e. even), but the odd part of its distribution
function~$\frac12 (\mu(x)-\mu(-x))$ gives a symmetric measure corresponding to the symmetric
moment sequence~\mh. In particular, the measure~$d\mu(x)$ is necessarily symmetric
if~$\mh\in\Detm_H$, and there are infinitely many symmetric solutions to the moment
problem~\eqref{eq:1} if~$\mh\in\IDetm_H$. Each Stieltjes moment
sequence~$\ms = (s_i)_{i=0}^\infty$ corresponds to a unique symmetric Hamburger moment
sequence~$(s_0,0,s_1,0,s_2,0,\dots)$ and conversely: if~$(s_0,0,s_1,0,s_2,0,\dots)$ is a
symmetric Hamburger moment sequence, then~\ms\ is a Stieltjes moment sequence. Moreover, these
sequences~$\ms$ and~$(s_0,0,s_1,0,s_2,0,\dots)$ are simultaneously determinate,
see~\cite{Chihara}.

\section{%
    \texorpdfstring{``Rigidity'' and ``nonrigidity'' of moment sequences}%
    {“Rigidity” and “nonrigidity” of moment sequences}}
\label{sec:rigid-nonr-moment}
\begin{theorem}\label{th:2}
    Given some positive integer~$n$, let
    \[
        \mh= (c_0,c_1,\dots,c_{2n-1},c_{2n},c_{2n+1},\dots)\quad\text{and}\quad
        \mt\coloneqq (c_0^*,c_1^*,\dots,c_{2n-1}^*,c_{2n},c_{2n+1},\dots)
    \]
    be two Hamburger moment sequences which differ in at most~$2n$ leading entries.
    If~$(c_i)_{i=2n}^\infty\in\Detm_H$, then~$\mh$ and~$\mt$ coincide up to their leading
    entries~$c_0$ and~$c_0^*$.

    In other words, if~$\mh\in\Detm_H$ and~$\ind_0(\mh)\ge n$, then necessarily
    \[
        c_1^*=c_1,\quad c_2^*=c_2,\quad\dots,\quad c_{2n-1}^*=c_{2n-1}.
    \]
\end{theorem}
\begin{proof}
    If~$(c_i)_{i=2n}^\infty\in\Detm_H$, then Lemma~\ref{lm:1}
    yields~$\mh,\mt\in\Detm_H$; thus, there exist two uniquely determined
    measures~$d\mu(x)$ and~$d\nu(x)$ corresponding to, resp.,~\mh\ and~$\mt$. Moreover,
    due to~$x^{2n}d\mu(x)=x^{2n}d\nu(x)$ these measures can differ only by a concentrated mass
    at the origin: $d\mu(x)=d\nu(x)+M\cdot\delta(x)\,dx$, where~$M\in\mathbb{R}$.
    Therefore,~$c_0 = M + c_0^*$, while the remaining entries in $\mh$ and~$\mt$
    coincide.
\end{proof}

Considering symmetric Hamburger moment sequences immediately gives:

\begin{corollary}\label{cr:1}
    Let~$n$ be some positive integer, let also
    \[
        \ms=(s_0,s_1,\dots,s_{n-1},s_{n},s_{n+1},\dots)
        \quad\text{and}\quad
        \mt\coloneqq (s_0^*,s_1^*,\dots,s_{n-1}^*,s_{n},s_{n+1},\dots)
    \]
    be two
    Stieltjes moment sequences which differ in at most~$n$ leading entries.
    If~$(s_i)_{i=n}^\infty\in\Detm_S$, then~$\ms$ and~$\mt$ coincide up to their leading
    entries~$s_0$ and~$s_0^*$.

    In other words, if~$\ms\in\Detm_S$ and~$\ind(\ms)\ge n$, then necessarily
    \[
        s_1^*=s_1,\quad s_2^*=s_2,\quad\dots,\quad s_{n-1}^*=s_{n-1}.
    \]
\end{corollary}

Theorem~\ref{th:2} shows that the Hamburger moment sequences with infinite index of determinacy
are rigid. The next theorem gives the contrary for indeterminate sequences and for determinate
sequences with a finite index of determinacy.

\begin{theorem}\label{th:3}
    Suppose that~$\mh=(c_i)_{i=0}^\infty\in\IDetm_H$ or~$\ind_0(\mh)< n$ for some positive
    integer~$n$. Then, for arbitrary real numbers~$c_{1}^*$, $c_{3}^*$,\dots, $c_{2n-1}^*$ and
    for any~$c_{2n}^*$ such that~$|c_{2n}^*-c_{2n}|<\varepsilon$ with a small
    enough~$\varepsilon$, there exists another Hamburger moment sequence
    \[
    \mt\coloneqq (c_0^*,c_1^*,\dots,c_{2n-1}^*,c_{2n}^*,c_{2n+1},\dots)\in\IDetm_H,
    \]
    where each of the even moments~$c_{2i}^*$, ~$i=0,1,\dots,n-1$, can be chosen in certain
    limits. More specifically,~$c_{2i}^*\in (b_{2i},+\infty)$, where~$b_{2i}>0$ depends on the
    values of~$c_{2i+1}^*,c_{2i+2}^*,\dots,c_{2n-1}^*,c_{2n},c_{2n+1},\dots$.
\end{theorem}

See the proofs in Section~\ref{sec:proofs1}. Observe that, like indeterminate Hamburger moment
sequences, the determinate sequences with~${\relpenalty=10000 \ind_0(\mh)=0}$ allow changing any
finite number of leading entries. The correspondence between symmetric and Stieltjes moment
sequences induces the following fact:

\begin{corollary}\label{cr:2}
    Let~$\ms=(s_i)_{i=0}^\infty\in\IDetm_S$ or~$\ind(\ms)\le n$. Then there is the Stieltjes
    moment sequence
    \[
    \mt\coloneqq (s_0^*,s_1^*,\dots,s_{n-1}^*,s_{n},s_{n+1},\dots)\in\IDetm_S,
    \]
    such that the moments~$s_{i}^*$, ~$i=0,1,\dots,n-1$, can be set to any numbers within
    certain limits. More specifically,~$s_i^*\in (b_{i},+\infty)$, where~$b_{i}>0$ depends on
    the values of~$s_{i+1}^*,s_{i+2}^*,\dots,s_{n-1}^*,s_n,s_{n+1},\dots$.
\end{corollary}

\section{A stronger version of the problem}
\label{sec:results-strong-vers}
The stronger version of the initial problem is to determine under which conditions ``nonrigid''
sequences allow arbitrary perturbations of the entries provided that the perturbations are small
enough. Here we introduce a few extensions in this direction of the facts stated in the previous
section.

It follows directly from the definition~\eqref{eq:1} that, if~\mh\ and~$\mt$ are two Hamburger
sequences corresponding, respectively, to the measures~$d\mu(x)$ and~$d\nu(x)$, then their
convex combination~$\eta\mh+(1-\eta)\mt$, where~$0<\eta<1$, is the Hamburger sequence which
corresponds to the measure~$\eta d\mu(x)+(1-\eta)d\nu(x)$. The simplest example that a sum of
two determinate sequences may be indeterminate can be found in the proof of Lemma~\ref{lm:2},
where one of the sequences is induced by the Dirac measure~$\delta(x)\,dx$, see
Section~\ref{sec:proofs1}. However, if at least one of the sequences~\mh\ and~$\mt$ is
indeterminate, then their convex combination~$\eta\mh+(1-\eta)\mt$ must also be indeterminate:
at least one of the summands in~$\eta d\mu(x)+(1-\eta)d\nu(x)$ is not determined uniquely.
Trimming leading entries of determinate sequences with finite indices immediately yields:
\begin{lemma}\label{lm:3}
    Let~\mh\ and~$\mt$ be two Hamburger moment sequences and let~$0<\eta<1$. Then
    \[
        \ind_0\left(\eta\mh+(1-\eta)\mt\right) \le \min\big\{\ind_0(\mh),\ind_0(\mt)\big\}
        .
    \]
\end{lemma}

Let the moment sequences~$\mh$ and~$\mt$ be as in Theorem~\ref{th:3} such
that~$\mt$ is indeterminate, then the
sequence~$(1-\varepsilon)\mh+\varepsilon\mt$ is also indeterminate for
each~$\varepsilon\in (0,1)$. At the same time, this latter sequence can be as close to~$\mh$ as
we want; therefore:
\begin{corollary}\label{cr:smooth_var}
    Every neighbourhood%
    \footnote{The neighbourhood is in the sense of certain norm, e.g. one of the norms of the
        spaces~$l_p$, ~$1\le p\le \infty$.} of each determinate Hamburger or Stieltjes moment
    sequence of a finite index contains an indeterminate moment sequence of the same type, which
    only differs in a finite number of entries.
\end{corollary}

The indeterminate sequences allow changing their entries quite freely:
\begin{theorem}\label{th:strong}
    If a Stieltjes or Hamburger moment sequence is indeterminate, then any small enough
    variation of a finite number of its entries gives an indeterminate moment sequence of the
    same type.
\end{theorem}

This is however not true for determinate sequences, even if the index of determinacy is finite:

\begin{theorem}\label{th:strong2}
    Let~$\mh=(c_0,c_1,c_2,\dots)$ be a Hamburger or Stieltjes moment sequence, and let~$m>0$ be
    an integer number. Then the sequence~$\mh$ is necessarily indeterminate ($\mh\in\IDetm_H$
    or, resp., $\mh\in\IDetm_S$), if all small enough variations of the entry~$c_{m}$ that
    change no other entries~$c_n$ with~$n\ne m$ result in moment sequences of the same type.
\end{theorem}

Let a determinate moment sequence~$\mh=(c_i)_{i=0}^{\infty}$ have a finite
index~$\ind_0(\mh)\eqqcolon m$. After comparison of the last theorem with Theorem~\ref{th:3} and
Corollary~\ref{cr:smooth_var}, one can ask whether it is possible to change a single moment
of~\mh\ provided that all other entries are fixed. This question lies outside the scope of the
present work, the answer however must depend on the sequence.

To illustrate, let us consider the basic case. By Theorem~\ref{th:3} and
Corollary~\ref{cr:smooth_var}, we may change the entries~$c_0,\dots,c_{2m-1},c_{2m}$ in certain
limits and keep~$c_{2m+1}$, $c_{2m+2},\dots$. At the same time, the entry~$c_{2m}$ can only be
increased, since~$(c_i)_{i=2m}^{\infty}\in\Detm_H$ and the corresponding measure has no atom at
the origin. More specifically, the point~$(c_{2m+1},c_{2m})$ lies on the border of the parabolic
region introduced in Lemma~\ref{lm:2}. That point cannot leave the region (in particular,
$c_{2m}$ cannot decrease), otherwise~\mh\ is no longer a Hamburger moment sequence. Now,
if~$(c_{2m+1},c_{2m})$ corresponds to the minimal value of~$c_{2m}$, then changing~$c_{2m+1}$ is
impossible on condition that all other entries are fixed. When the value of~$c_{2m}$ is not
minimal, the moment~$c_{2m+1}$ can be either increased or decreased depending on which part of
the border of the involved parabolic region the point~$(c_{2m+1},c_{2m})$ belongs to.

\section{Proofs of Theorems~\ref{th:3} and~\ref{th:strong}}\label{sec:proofs1}
We rely on the following fact (see its proof below):
\begin{lemma}\label{lm:2}
    Let~$\mh=(c_0,c_1,\dots)\in\IDetm_H$. Then there exists another Hamburger moment sequence
    \[
        \mt\coloneqq (c_{-2},c_{-1},c_{0},c_{1},\dots)
    \]
    for any given real number~$c_{-1}$. Moreover,~$c_{-2}$ then can be set to any number
    satisfying
    \begin{equation}\label{eq:4}
        c_{-2}\ge
        c_{-1}^2 \sum_{k=0}^\infty P_k^2(0)+ \sum_{k=0}^\infty Q_k^2(0) + 2c_{-1}\sum_{k=0}^\infty P_k(0)Q_k(0),
    \end{equation}
    here $P_k(z)$ is the $k$th orthonormal polynomial related to~\mh\ and $Q_k(z)$ is the
    corresponding polynomial of the second kind. This inequality is strict if and only if the
    new sequence~$\mt$ is indeterminate.
\end{lemma}
\begin{remark}
    It is noteworthy that the right hand side of the last inequality for~$c_{-2}$ reaches its
    minimum
    \[
    \begin{aligned}
        c_{-2}^* &=\sum_{k=0}^\infty Q_k^2(0) - \rho(0)\left(\sum_{k=0}^\infty
            P_k(0)Q_k(0)\right)^2
        \quad\text{when~$c_{-1}$ turns to}\\
        c_{-1}^* &= -\frac{\sum_{k=0}^\infty P_k(0)Q_k(0)}{\sum_{k=0}^\infty P_k^2(0)}
        =-\rho(0)\sum_{k=0}^\infty P_k(0)Q_k(0),
    \end{aligned}
    \]
    where~$\rho(0)$ stands for the maximal mass that can be put at the origin among all
    distributions corresponding to~\mh\ (the related measure is labelled in the proof below
    by~$d\mu(x;\infty)$).

    In the terms of the expressions~\eqref{eq:wfN} and~\eqref{eq:c_m2_t}, the minimal value
    of~$c_{-2}$ can be expressed as
    \[
    c_{-2}^* = a'(0) - \frac{b'(0)c'(0)}{d'(0)}
    \quad \text{for}\quad
    c_{-1}^*
    =-\frac{b'(0)+c'(0)}{2d'(0)}
    =-\frac{b'(0)}{d'(0)}.
    \]
    Another option is that these formulae can be expressed through the limits of certain
    determinants built of the moments, see e.g.~\cite[p.~181]{HamburgerT3}
    or~\cite[pp.~66--67]{ShohatTamarkin}. Limits of similar determinants appear in our
    polynomial~$p_{*}(\gamma)$ from the proof of Lemma~\ref{lm:strong2b} for~$m=1$: the
    condition~$p_{*}(\gamma)\ge 0$ is, in fact, the analogue of~\eqref{eq:4}.
\end{remark}

When~$\mh=(c_{i})_{i=0}^\infty\in\IDetm_H$, there exists the corresponding Nevanlinna
parameterization (see e.g.~\cite[Ch.~II \S 4]{Akhiezer}) --- four real entire
functions~$a(z),b(z),c(z),d(z)$ of at most minimal type of exponential order which satisfy the
formula%
\footnote{Note that the Nevanlinna parameterization in~\cite{Akhiezer} uses the reciprocal of
    our parameter~$t$; our version is closer to~\cite{ShohatTamarkin}.}
\begin{equation}
    m(z;t)\coloneqq
    \int_{-\infty}^\infty \frac{d\mu(x;t)}{x-z}
    =-\frac{a(z)-tc(z)}{b(z)-td(z)}
    \qquad
    \left(m(z;\infty)=-\frac{c(z)}{d(z)}\right)\label{eq:wfN}
\end{equation}
with the parameter~$t$ running over~$\mathbb{R}\cup\{\infty\}$. Each of the measures~$d\mu(x;t)$
solves the moment problem~\eqref{eq:1}, that is
\[
c_{i} = \int_{-\infty}^\infty x^i \, d\mu(x;t) \quad
\text{for all}\quad
i=0,1,\dots
\]
independently of~$t\in\mathbb{R}\cup\{\infty\}$. These measures are called \emph{N-extremal},
because the polynomials are dense in the spaces~$L_2\big(\mu(\,\cdot\;;t)\big)$. The whole set
of the solutions of the moment problem~\eqref{eq:1} comes from substituting~$(-t)$
in~\eqref{eq:wfN} by all mappings of the upper half-plane into itself. At the origin, these four
functions satisfy
\begin{equation}\label{eq:abcd_origin}
    b(0)=-1,\quad c(0)=1
    \quad\text{and}\quad
    a(0)=d(0)=0.
\end{equation}

\begin{proof}[Proof of Lemma~\ref{lm:2}]
    Suppose that~$\mh\in\IDetm_H$ and the corresponding N-extremal measure~$d\mu(x;t)$ is as
    in~\eqref{eq:wfN}. We confine ourselves with finite real values of the parameter~$t$. Under
    this condition, the support of~$d\mu(x;t)$ has no points in a neighbourhood of the origin,
    i.e. if~$|t|<t_{\text{max}}$, then~$(-\epsilon,\epsilon)\notin\supp d\mu(x;t)$ for
    some~$\epsilon>0$ depending on~$t_{\text{max}}$. Thus, the expression~$x^{-2}d\mu(x;t)$ also
    determines a measure with the same support and the moments
    \begin{equation}\label{eq:2}
        c_{-2}(t) = \int_{-\infty}^\infty \frac{d\mu(x;t)}{x^2}, \quad
        c_{-1}(t) = \int_{-\infty}^\infty \frac{d\mu(x;t)}{x}, \quad
        c_{0},c_{1},c_{2},\dots
        .
    \end{equation}
    % see e.g.~\cite{HamburgerT3},~\cite[pp.~54--56]{Akhiezer}.
    The function~$m(z;t)$ satisfies
    \[
    \begin{aligned}
        m(z;t)
        &= \int_{-\infty}^\infty \frac 1{1-\frac zx}\cdot \frac{d\mu(x;t)}x
        = \int_{-\infty}^\infty \sum_{k=0}^\infty\left(\frac zx\right)^k \!\cdot \frac{d\mu(x;t)}x\\
        &= \int_{-\infty}^\infty \frac{d\mu(x;t)}x
        + z \int_{-\infty}^\infty \frac{d\mu(x;t)}{x^2} + O(z^2)\\
        &= c_{-1}(t)
        + z c_{-2}(t) + O(z^2)
        .
    \end{aligned}
    \]
    The expression on the right-hand side coincides with the Taylor expansion of~$m(z;t)$ near
    the origin. Consequently,
    \[
        c_{-1}(t) = m(0;t)\quad\text{and}\quad
        c_{-2}(t) = \frac {dm}{dz}(0;t).
    \]
    Then the equalities~\eqref{eq:abcd_origin} yield
    \begin{align}\label{eq:c_m_t}
        c_{-1}(t)
        &=-\frac{a(0)-tc(0)}{b(0)-td(0)}
        =-\frac{0-t}{-1-0t} = -t
        \quad\text{and}
      \\
      \label{eq:c_m2_t}
            c_{-2}(t)
            &= -\frac{\big(a'(0)-tc'(0)\big)\big(b(0)-td(0)\big)
              -\big(a(0)-tc(0)\big)\big(b'(0)-td'(0)\big)}{\big(b(0)-td(0)\big)^2}\\
      \nonumber
            &= a'(0) - tc'(0) - t\big(b'(0)-td'(0)\big)
            = t^2 d'(0) + a'(0) - t\big(c'(0)+b'(0)\big).
    \end{align}

    Our consideration almost literally repeats~\cite[p.~181]{HamburgerT3}, we therefore can
    obtain that the moment sequence~\eqref{eq:2} is definite in the same way. Let us however
    relate it to M.~Riesz's Theorem on N\nobreakdash-extremality and completeness of orthogonal
    polynomials, see e.g.~\cite[p.~2796, p.~2801]{BergDuran} or~\cite[p.~55]{Akhiezer}.
    Let~$\big(P_k(z)\big)_{k=0}^\infty$ be the sequence of the orthonormal polynomials induced
    by~$\mh$ or, which is the same, by~$d\mu(x;t)$ for
    any~$t\in(-t_{\text{max}},t_{\text{max}})$; let~$\big(Q_k(z)\big)_{k=0}^\infty$ denote the
    corresponding polynomials of the second kind. Given any non-real number~$z$, Parseval's
    equality for~$\left(x-z\right)^{-1}$ has the form:
    \[
    \frac{m(z;t)-m(\overline z;t)}{z-\overline z}
    = \int_{-\infty}^\infty \frac{d\mu(x;t)}{|x-z|^2}
    =\sum_{k=0}^\infty \big|m(z;t)P_k(z)+ Q_k(z)\big|^2
    ,
    \]
    see~\cite[p.~40]{Akhiezer}. The series on the right hand side converge due to the
    indeterminacy of~\mh. The limit of this equality as~$z\to 0$ exists due
    to~$0\notin\supp d\mu(x;t)$; it is equal to
    \[
    \begin{aligned}
        c_{-2}(t)=\frac{dm}{dz}(0;t)
        &=\sum_{k=0}^\infty \big|m(0;t)P_k(0)+ Q_k(0)\big|^2\\
        &=c_{-1}^2(t) \sum_{k=0}^\infty P_k^2(0)+ \sum_{k=0}^\infty Q_k^2(0)
        + 2c_{-1}(t)\sum_{k=0}^\infty P_k(0)Q_k(0),
    \end{aligned}
    \]
    which is another form of~\eqref{eq:c_m2_t}, cf.~\cite[p.~54]{Akhiezer}. By~\eqref{eq:c_m_t},
    each value of~$c_{-1}(t)$ corresponds to a single N-extremal measure~$d\mu(x;t)$. At the
    same time, the choice of~$c_{-2}(t)$ as in~\eqref{eq:c_m2_t} is equivalent to Parseval's
    equality and, hence, to the N-extremality.%
    \footnote{Under our conditions, Parseval's equality at the origin implies Parseval's
        equality everywhere outside the real line by M.~Riesz's Theorem~\cite[p.~43]{Akhiezer}:
        its proof extends to the origin with minimal changes if~$0\notin\supp d\mu(x;t)$.}
    So, the moment sequence~\eqref{eq:2} determines a unique measure. If a feasible choice
    of~$c_{-1}$ and~$c_{-2}$ does not fix a unique measure, then Parseval's equality turns to
    Bessel's inequality~\eqref{eq:4}.

    More specifically, let~$c_{-1}$ and~$c_{-2}$ be given real numbers. If the equality
    in~\eqref{eq:4} holds, then~$c_{-2}=c_{-2}(-c_{-1})$ by~\eqref{eq:c_m2_t} and, hence, the
    sequence~$\mt$ is determinate. Suppose that~$c_{-2}$ is an arbitrary real number satisfying
    the strict inequality in~\eqref{eq:4}. Thus, there exists~$\varepsilon>0$ such that both
    points~$(c_{-1},c_{-2})$ and~$(c_{-1},c_{-2}-\varepsilon)$ lie strictly inside the parabolic
    region determined by~\eqref{eq:4}. Then the quadratic equation
    \[
    t^2 \sum_{k=0}^\infty P_k^2(0) - 2t\sum_{k=0}^\infty P_k(0)Q_k(0) + \sum_{k=0}^\infty Q_k^2(0)=c_{-2}
    \]
    has two distinct real solutions, say~$t_{1}, t_{2}$. The same is true for the equation
    \[
    t^2 \sum_{k=0}^\infty P_k^2(0) - 2t\sum_{k=0}^\infty P_k(0)Q_k(0) + \sum_{k=0}^\infty Q_k^2(0)=c_{-2}-\varepsilon,
    \]
    whose solutions we denote by~$t_{3}, t_{4}$. Note that the points~$(t_1,c_{-2})$
    and~$(t_2,c_{-2})$, as well as $(t_3,c_{-2}-\varepsilon)$ and~$(t_4,c_{-2}-\varepsilon)$,
    are on the boundary of the above parabolic region. Since this region is convex, there
    exist~$\vartheta,\eta\in(0,1)$ such
    that~$c_{-1}=\vartheta t_{1}+(1-\vartheta)t_2=\eta t_{3}+(1-\eta)t_4$. Then the measure
    \[
    d\nu_1(x)
    \coloneqq\vartheta\frac{d\mu(x;t_1)}{x^2} +(1-\vartheta)\frac{d\mu(x;t_2)}{x^2}
    \]
    has the following moments:
    \[
    \begin{aligned}
        \int_{-\infty}^\infty d\nu_1(x)= \vartheta c_{-2}(t_1)+(1-\vartheta)c_{-2}(t_2)
        =c_{-2}(t_1)&= c_{-2},\\
        \int_{-\infty}^\infty x\,d\nu_1(x) = \theta t_1 + (1-\theta)t_2 &=c_{-1},\\
        \theta c_{k} + (1-\theta)c_{k}&=c_{k}\quad\text{for }k=0,1,2,\dots.
    \end{aligned}
    \]
    The measure
    \[
    d\nu_2(x)
    \coloneqq\eta\frac{d\mu(x;t_3)}{x^2} +(1-\eta)\frac{d\mu(x;t_4)}{x^2} + \varepsilon\delta(x)\,dx
    \]
    in turn has the same sequence of moments:
    \[
    \begin{aligned}
        \eta c_{-2}(t_3)+(1-\eta)c_{-2}(t_4) + \varepsilon = c_{-2}-\varepsilon + \varepsilon&= c_{-2},\\
        \eta t_3 + (1-\eta)t_4 &= c_{-1},\\
        \eta c_{k} + (1-\eta)c_{k} &= c_{k}\quad\text{for }k=0,1,2,\dots.
    \end{aligned}
    \]
    However, the supports of~$d\nu_1(x)$ and~$d\nu_2(x)$ have no common points, see e.g. the
    footnote in~\cite[p.~55]{Akhiezer}. Thus, the moment
    sequence~$\mt = (c_{-2}, c_{-1}, c_{0}, c_{1},\dots)$ is indeterminate. \end{proof}

\begin{proof}[Proof of Theorem~\ref{th:3}]
    Since~$\mh\in\IDetm_H$ or~$\ind_0(\mh)< n$, the sequence~$(c_{i})_{i=2n}^\infty$ is
    indeterminate. Then there is a corresponding N-extremal measure, which has the
    atom~$\varepsilon\delta(x)\,dx$ at the origin. Thus, each
    sequence~$(c_{2n}^*,c_{2n+1},\dots)$ belongs to~$\IDetm_H$ provided that the positive
    number~$c_{2n}^*$ satisfies~$|c_{2n}^*-c_{2n}|<\varepsilon$.

    % \xxtextls{}
    Lemma~\ref{lm:2} implies that, for an arbitrary real number~$c_{2n-1}^*$, there exists
    some~$b_{2n-2}$ such that
    \[
        (c_{2n-2}^*,c_{2n-1}^*,c_{2n}^*,c_{2n+1},\dots)\in\IDetm_H
        \quad\text{as soon as}\quad
        c_{2n-2}^*>b_{2n-2}
        .
    \]
    Clearly, we can fix~$c_{2n-2}^*$ so that additionally~$c_{2n-2}^*\ne c_{2n-2}$. Analogously,
    if we already have
    \[
        (c_{2n-2k}^*,\dots,c_{2n-1}^*,c_{2n}^*,c_{2n+1},\dots)\in\IDetm_H
    \]
    for some integer~$k$, ~$1\le k < n$, then Lemma~\ref{lm:2} yields that for an arbitrary real
    number~$c_{2n-k-1}^*$ there exists some~$b_{2n-2k-2}$ such that
    \[
        (c_{2n-2k-2}^*,\dots,c_{2n-1}^*,c_{2n}^*,c_{2n+1},\dots)\in\IDetm_H
        \quad\text{for}\quad
        % \ne c_{2n-2k-2}
        c_{2n-2k-2}^*> b_{2n-2k-2}
        .
    \]
    By induction, the theorem is therefore true.
\end{proof}

\begin{proof}[Proof of Theorem~\ref{th:strong}]
    The statement of this theorem only uses the ``distance'' between those sequences that differ
    in a finite number of entries. Since all norms in~$\mathbb{R}^n$ are equivalent, it is
    enough to prove the theorem for the~$l_\infty$ norm, i.e. the supremum of absolute values of
    the entries.
    
    Let~$n+1$ be the number of leading entries in an indeterminate Hamburger moment
    sequence~$\mh$, that vary. To prove by induction on~$n$, consider firstly the base
    case~$n=0$. Since~$\mh\in\IDetm_H$, there exists the corresponding N-extremal
    measure~$d\mu(x)$ with an atom~$\varepsilon_0\delta(x)\,dx$ at the
    origin,~$\varepsilon_0>0$. Therefore,~$d\mu(x)+\gamma\delta(x)\,dx$ is a positive measure
    with the moments~$c_0+\gamma, c_1,c_2,\dots$ provided
    that~$-\varepsilon_0<\gamma<\varepsilon_0$. As a result,~$(f_0,c_1,c_2,\dots)\in\IDetm_H$
    if~$|f_0-c_0|<\varepsilon_0$.

    Now assume that the following property is satisfied for~$n-1\ge 0$:
    \[
        \max_{0\le i<n}|f_i-c_i|<\varepsilon_{n-1}
        \implies
        (f_0,f_1,\dots,f_{n-1},c_{n},c_{n+1},\dots)\in\IDetm_H.
    \]
    For the inductive step, we need to show that there exists a number~$\varepsilon_{n}>0$ such
    that
    \begin{equation}\label{eq:3}
        \max_{0\le i<n+1}|f_i-c_i|<\varepsilon_{n}
        \implies
        (f_0,f_1,\dots,f_{n-1},f_{n},c_{n+1},c_{n+2},\dots)\in\IDetm_H.
    \end{equation}
    
    Recall that~$\mh\in\IDetm_H$. For certain small enough~$\varepsilon>0$, if we
    take~$d_n\coloneqq c_n+\varepsilon$ and~$e_n\coloneqq c_n-\varepsilon$, then by
    Theorem~\ref{th:3} there exist real numbers~$d_0,d_1,\dots,d_{n-1}$
    and~$e_0,e_1,\dots,e_{n-1}$ such that both
    \[
        \mathfrak{d}\coloneqq
        (d_0,d_1,\dots,d_n,c_{n+1},c_{n+2},\dots)
        \quad\text{and}\quad
        \mathfrak{e}\coloneqq
        (e_0,e_1,\dots,e_n,c_{n+1},c_{n+2},\dots)
    \]
    are indeterminate Hamburger moment sequences. Put
    \[
        \vartheta\coloneqq \frac{\varepsilon_{n-1}/2}
          {\max\limits_{0\le i< n}\max\big\{|d_i-c_i|,|e_i-c_i|,1\big\}}
        \quad\text{and}\quad
        \varepsilon_n\coloneqq
          \min\left\{\frac{\varepsilon_{n-1}}4,\frac{\varepsilon}2\right\}
    \]
    and suppose that the real numbers~$f_0,f_1,\dots,f_n$
    satisfy~$\max_{0\le i<n+1}|f_i-c_i|<\varepsilon_{n}$.
    
    % Without loss of generality,
    Assume that~$f_n<c_n$. Then the number
    \[
    \alpha\coloneqq \frac{c_n- f_n}{\vartheta (c_n- e_n)} <
    \frac{\varepsilon/2}{\vartheta\varepsilon} \in\Big[0,\frac 12\Big)
    \]
    is such that
    \[
        f_n
        = c_n - \alpha\vartheta (c_n- e_n)
        =\alpha\vartheta e_n+(1-\alpha\vartheta) c_n
        .
    \]
    At the same time, the absolute value
    of~$\zeta_i \coloneqq f_i- c_i + \alpha\vartheta (c_i- e_i)$ for~$i<n$ satisfies
    \[
        |\zeta_i| \le |f_i- c_i| + \alpha\vartheta |c_i- e_i|
        < \varepsilon_n + \frac{1}2\alpha\varepsilon_{n-1}
        < \frac{1}4 \varepsilon_{n-1} + \frac{1}4 \varepsilon_{n-1}
        =\frac{1}2\varepsilon_{n-1}.
    \]
    Consequently,  for~$i<n$
    \[
        \begin{aligned}
        f_i
        &= c_i - \alpha\vartheta (c_i- e_i) + f_i- c_i + \alpha\vartheta (c_i- e_i)
        \\
        &= \alpha\vartheta e_i + (1- \alpha\vartheta) c_i + \zeta_i
        = \alpha\vartheta e_i + (1- \alpha\vartheta)
            \left(c_i + \frac{\zeta_i}{1- \alpha\vartheta}\right)
            .
        \end{aligned}
    \]
    Since
    \[
        \left|\frac{\zeta_i}{1- \alpha\vartheta}\right|
        <2\cdot\frac{1}2 \varepsilon_{n-1}
        =
        \varepsilon_{n-1},
    \]
    the sequence
    \[
    \mathfrak g\coloneqq \left(c_0 + \frac{\zeta_0}{1- \alpha\vartheta},%
        c_1 + \frac{\zeta_1}{1-\alpha\vartheta},\dots,%
        c_{n-1} + \frac{\zeta_{n-1}}{1-\alpha\vartheta},%
        c_n,c_{n+1},\dots\right)
    \]
    in an indeterminate Hamburger moment sequence by the induction hypothesis. In other words,
    we obtainded the representation
    \[
        (f_0,f_1,\dots,f_{n-1},f_{n},c_{n+1},c_{n+2},\dots)
        = \alpha\vartheta\mathfrak{e} + (1-\alpha\vartheta)\mathfrak{g},
    \]
    where both terms on the right-hand side, and hence the left-hand side, belong
    to~$\IDetm_H$. This representation yields~\eqref{eq:3} in the case~$f_n<c_n$. The case
    when~$f_n>c_n$ follows analogously after replacing~$\mathfrak{e}$ by~$\mathfrak{d}$, so the
    theorem is therefore true for Hamburger moment sequences. The assertion on Stieltjes moment
    sequences follow from considering the related symmetric Hamburger sequence.
\end{proof}

\section{Proof of Theorem~\ref{th:strong2}}\label{sec:proofs2}
\begin{lemma}\label{lm:strong2a}
    Suppose that~$\mh=(c_0,c_1,c_2,\dots)$ is a determinate Hamburger moment sequence and that a
    real number~$\varepsilon>0$ and an integer~$m>0$ are such that
    \[
    \mt\coloneqq (c_0,\dots,c_{m-1},c_{m}^*,c_{m+1}\dots)
    \]
    is a Hamburger moment sequence for all~$c_{m}^*$ satisfying~$|c_{m}-c_{m}^*|<\varepsilon$.
    Then~\mt\ is also determinate and~$\ind_0(\mt) = \ind_0(\mh)<\frac m2$.
\end{lemma}

\begin{proof}
    Denote~$r\coloneqq\ind_0(\mh)$ and~$q=q(c_{m}^*)\coloneqq\ind_0(\mt)$. On the one hand,
    Theorem~\ref{th:2} implies that~$2r \le m$ and~$2q \le m$. If~$m$ is odd, both inequalities
    are clearly strict. Let us show that they are also strict on the assumption that~$m$ is
    even. Indeed, for each fixed~$c_{m}^*\in(c_{m}-\varepsilon,c_{m}+\varepsilon)$, the trimmed
    sequence
    \[
        \mh_m\coloneqq (c_{m}^*,c_{m+1},c_{m+2},\dots)
    \]
    is either indeterminate or determinate of index~$0$ by the definition of index. The latter
    condition for a Hamburger sequence implies that the corresponding measure vanishes near the
    origin and adding a point mass at the origin makes it indeterminate (a straightforward
    consequence of e.g. Lemma~\ref{lm:2}). In other words, if~$\mh_m\in\Detm_H$ and its index is
    zero, then
    \[
        (c_{m}^*-\widetilde\varepsilon,c_{m+1},c_{m+2},\dots)
    \]
    cannot be a Hamburger sequence when~$\widetilde\varepsilon>0$. This however fails to be true
    for any~$\widetilde\varepsilon<\varepsilon-(c_{m}-c_{m}^*)$; therefore,~$\mh_m\in\IDetm_H$
    as soon as~$|c_{m}-c_{m}^*|<\varepsilon$, which means~$2r < m$ and~$2q(c_{m}^*) < m$ for
    even~$m$.

    On the other hand, the trimmed sequence~$(c_{i})_{i=2r}^\infty$ can be expressed as the
    following convex combination of two sequences:
    \begin{align*}
      (c_{2r},\dots,c_{m-1},c_{m},c_{m+1}\dots)
      ={}& \frac 12 (c_{2r},\dots,c_{m-1},c_{m}^*,c_{m+1}\dots)\\
         &+ \frac 12 (c_{2r},\dots,c_{m-1},2c_{m}-c_{m}^*,c_{m+1},\dots)
        .
    \end{align*}
    Both sequences on the right-hand side are the Hamburger sequences provided that
    \[
    \big|c_{m}-(2c_{m}-c_{m}^*)\big|=\big|c_{m}-c_{m}^*\big|<\varepsilon,
    \]
    so they are determinate, because their sum~$(c_{i})_{i=2r}^\infty$ is determinate --- see
    the explanation before Lemma~\ref{lm:3}. Consequently,~$r\le q(c_{m}^*)$.

    Now, for each~$c_{m}^*\in(c_{m}-\frac \varepsilon 2,c_{m}+\frac \varepsilon 2)$ the
    left-hand side of
    \[
        \begin{aligned}
            2(c_{2q},\dots,c_{m-1},c_{m}^{*},c_{m+1}\dots)
            ={} & (c_{2q},\dots,c_{m-1},c_{m},c_{m+1}\dots)
            \\
            &+ (c_{2q},\dots,c_{m-1},2c_{m}^{*}\!-c_{m},c_{2m+1},\dots)
        \end{aligned}
    \]
    is a determinate moment sequence. So, both sequences on its right-hand side must be
    determinate, and hence~$q(c_{m}^*)\le r$. Analogously, for
    each~$c_{m}^{**}\in(c_{m}^*-\frac \varepsilon 4,c_{m}^*+\frac \varepsilon 4)$ the identity
    \[
        \begin{aligned}
        2(c_{2q},\dots,c_{m-1},c_{m}^{**},c_{m+1},\dots)
        ={}&(c_{2q},\dots,c_{m-1},c_{m}^*,c_{m+1},\dots)\\
        &+ (c_{2q},\dots,c_{m-1},2c_{m}^{**}\!-c_{m}^*,c_{m+1},\dots)
        \end{aligned}
    \]
    shows that both sequences on its right-hand side must be determinate, and
    hence~$q(c_{m}^{**})\le q(c_{m}^{*})\le r$. In other words, we have~$q(c_{m}^{*})\le r$
    whenever~$c_{m}^{*}\in(c_{m}-\frac 34\varepsilon,c_{m}+\frac 34\varepsilon)$. The same
    manipulations can be continued further, which yields the inequality~$q(c_{m}^{*})\le r$ for
    all~$c_{m}^{*}\in(c_{m}-\varepsilon,c_{m}+\varepsilon)$. It implies that~$q=q(c_{m}^{*})= r$
    being combined with the reverse inequality~$r\le q(c_{m}^*)$ obtained above.
\end{proof}

\begin{lemma}\label{lm:strong2b}
    No Hamburger moment sequence~$\mh\in\Detm_H$ with~$\ind_0(\mh)=0$ can satisfy the conditions
    of Lemma~\ref{lm:strong2a}.
\end{lemma}

The proof of this lemma is based on relations between certain determinants of the corresponding
Hankel matrix. More specifically, given a sequence~$\mathfrak e=(e_0,e_1,e_2,\dots)$ and two
integer numbers~$n>0$ and~$k\ge 0$ denote
\[
\Delta_{n}^{(k)}[\mathfrak e]\coloneqq
\begin{vmatrix}
    e_{2k}&e_{2k+1}&e_{2k+2}&\hdots&e_{2k+n-1}\\
    e_{2k+1}&e_{2k+2}&e_{2k+3}&\hdots&e_{2k+n}\\
    e_{2k+2}&e_{2k+3}&e_{2k+4}&\hdots&e_{2k+n}\\
    \vdots&\vdots&\vdots&\ddots&\vdots\\
    e_{2k+n-1}&e_{2k+n}&e_{2k+n+1}&\hdots&e_{2k+2n-2}
\end{vmatrix}
.
\]
If additionally~$0\le i<k$ and~$0\le j<k$, then
\[
f_{i,j}^{(k)}[\mathfrak e;n]\coloneqq
\begin{vmatrix}
    e_{i+j}&e_{i+k}&e_{i+k+1}&\hdots&e_{i+k+n-1}\\
    e_{j+k}&e_{2k}&e_{2k+1}&\hdots&e_{2k+n-1}\\
    e_{j+k+1}&e_{2k+1}&e_{2k+2}&\hdots&e_{2k+n}\\
    \vdots&\vdots&\vdots&\ddots&\vdots\\
    e_{j+k+n-1}&e_{2k+n-1}&e_{2k+n}&\hdots&e_{2k+2n-2}
\end{vmatrix}
=
\begin{vmatrix}
    e_{i+j}&e_{i+k}&\hdots&e_{i+k+n-1}\\
    e_{j+k}&&&\\
    \vdots&&&\\
    e_{j+k+n-1}&\multicolumn{3}{c}{\smash{\raisebox{\normalbaselineskip}
            {\Large$\Delta_{n}^{(k)}[\mathfrak e]$}\quad}}
\end{vmatrix}
.
\]
Since determinants are invariant under the transposition, the
matrix
\[
F_n^{(k)}[\mathfrak e]\coloneqq \big(f_{i,j}^{(k)}[\mathfrak e;n]\big)_{i,j=0}^{k-1}
\]
is symmetric.

\begin{proof}[Proof of Lemma~\ref{lm:strong2b}]
    Put~$\gamma\coloneqq c_{m}^*-c_{m}$ so that~$c_{m}^*=c_{m}+\gamma$. Denote also
    \[
    f_{i,j}[\mt] \coloneqq f_{i,j}^{(m+1)}[\mt;n-m-1]
    \]
    for the sake of brevity. The quadratic form
    \begin{equation}\label{eq:quadr_form_gamma}
        \sum_{i,j=0}^{\infty}c_{i+j}x^ix^j + \gamma\sum_{k=0}^{m}x^kx^{m-k}
    \end{equation}
    corresponding to~\mt\ is positive definite, see Remark~\ref{rem:defin-basic-facts};
    hence,~$\Delta_{n-m-1}^{(m+1)}[\mt]=\Delta_{n-m-1}^{(m+1)}[\mh]>0$. Sylvester's determinant
    identity~\cite[Chapter~I, \S 2]{GantmacherKrein} for the minors introduced above can be
    written as
    \[
    \Delta_{n}^{(0)}[\mt]
    =
    \left(\Delta_{n-m-1}^{(m+1)}[\mt]\right)^{-m}
    \cdot
    \det F_{n-m-1}^{(m+1)}[\mt]
    .
    \]
    On the right-hand side, only the entries \( f_{k,m-k}[\mt] \) with~$k=0,1,\dots,m$ depend
    on~$c_m^*$. More specifically,
    \[
        \begin{aligned}
            f_{k,m-k}[\mt]={}&
            \begin{vmatrix}
                c_{m}&c_{k+m+1}&\hdots&c_{k+n-1}\\
                c_{2m-k+1}&&&\\
                \vdots&&&\\
                c_{m-k+n-1}&\multicolumn{3}{c}{\smash{\raisebox{\normalbaselineskip}
                        {\Large$\Delta_{n-m-1}^{(m+1)}[\mh]$}\quad}}
            \end{vmatrix}
            +
            \begin{vmatrix}
                \gamma&0&\hdots&0\\
                c_{2m-k+1}&&&\\
                \vdots&&&\\
                c_{m-k+n-1}&\multicolumn{3}{c}{\smash{\raisebox{\normalbaselineskip}
                        {\Large$\Delta_{n-m-1}^{(m+1)}[\mh]$}\quad}} 
            \end{vmatrix}
            \\
            ={}& f_{k,m-k}[\mh] + \gamma\,\Delta_{n-m-1}^{(m+1)}[\mh]
            .
        \end{aligned}
    \]
    As a consequence,
    \[
    F_{n-m-1}^{(m+1)}[\mt]
    =
    \left( f_{i,j}[\mt] \right)_{i,j=0}^m
    =
    \left( f_{i,j}[\mh]
        + \gamma \delta_{i+j,m} \,\Delta_{n-m-1}^{(m+1)}[\mh] \right)_{i,j=0}^m
    ,
    \]
    where
    \[
        \delta_{i+j,m} =
        \begin{cases}
            0,&\text{ if } i+j\ne m;\\
            1,&\text{ if } i+j = m
        \end{cases}
    \]
    is the Kronecker delta. The Hadamard inequality~\cite[Chapter~I, \S 8]{GantmacherKrein}
    implies
    \begin{equation}\label{eq:Hadamard_ineq}
        f_{i,i}[\mh]\le c_{2i}\cdot \Delta_{n-m-1}^{(m+1)}[\mh]
        \quad\text{for}\quad
        i=0,1,\dots,m.
    \end{equation}
    All factors here are strictly positive as the principal minors of the Hankel matrix
    corresponding to~\eqref{eq:quadr_form_gamma} with~$\gamma=0$.
    % At the same time, for any~$0\le i<j<m$ we have
    Moreover,
    \[
        \begin{vmatrix}
            f_{i,i}[\mh]&f_{i,j}[\mh]\\[7pt]
            f_{j,i}[\mh]&f_{j,j}[\mh]
        \end{vmatrix}
        =
        \Delta_{n-m-1}^{(m+1)}[\mh]\cdot
        \begin{vmatrix}
            c_{2i}&c_{i+j}&c_{m+i}&\dots&c_{i+n-1}\\
            c_{i+j}&c_{2j}&c_{m+j}&\dots&c_{j+n-1}\\
            c_{m+i}&c_{m+j}&&&\\
            \vdots&\vdots&&&\\
            c_{i+n-1}&c_{j+n-1}&\multicolumn{3}{c}{\smash{\raisebox{\normalbaselineskip}
                    {\Large$\Delta_{n-m-1}^{(m+1)}[\mh]$}\quad}}
        \end{vmatrix}
        >0
    \]
    for any~$0\le i<j<m$ by Sylvester's determinant identity. Since the
    matrix~$F_{n-m-1}^{(m+1)}[\mh]$ is symmetric, the last inequality can be rewritten with the
    help of~\eqref{eq:Hadamard_ineq} as
    \[
    \left(f_{i,j}[\mh]\right)^2
    <
    f_{i,i}[\mh]\cdot f_{j,j}[\mh]
    \le
    c_{2i}c_{2j}\left(\Delta_{n-m-1}^{(m+1)}[\mh]\right)^2\!.
    \]
    Consequently, all entries of the matrix
    \[
    \left(
        \frac{ \vphantom{\Big|}
            f_{i,j}[\mt]
        }{ \vphantom{\Big|}
            \Delta_{n-m-1}^{(m+1)}[\mt]}
    \right)_{i,j=0}^{m}
    =
    \left(
        \frac{ \vphantom{\Big|}
            f_{i,j}[\mh]
        }{ \vphantom{\Big|}
            \Delta_{n-m-1}^{(m+1)}[\mh]}
        + \gamma \delta_{i+j,m}
    \right)_{i,j=0}^{m}
    \]
    are bounded in absolute value uniformly in~$n$ and~$\gamma\in(-\varepsilon,\varepsilon)$.
    Moreover, its determinant
    \[
    \begin{aligned}
        p_n(\gamma)
        \coloneqq{}&
        \frac{\det F_{n-m-1}^{(m+1)}[\mt]}{\left(\Delta_{n-m-1}^{(m+1)}[\mt]\right)^{m+1}}
        =\det
        \left(
            \frac{ \vphantom{\Big|}
                f_{i,j}[\mh]
            }{ \vphantom{\Big|}
                \Delta_{n-m-1}^{(m+1)}[\mh]}
            + \gamma \delta_{i+j,m}
        \right)_{i,j=0}^{m}
        \\
        ={}&(-1)^{\frac{m(m+1)}2}\gamma^{m+1}+\dots
        + \frac{\det F_{n-m-1}^{(m+1)}[\mh]}{\left(\Delta_{n-m-1}^{(m+1)}[\mh]\right)^{m+1}}\gamma^0
        .
    \end{aligned}
    \]
    is a polynomial in~$\gamma$ of degree~$m+1$. Its coefficients remain bounded uniformly
    in~$n$, as certain sums of products of at most~$m+1$ bounded factors.

    The~$m+2$-dimensional bounded set of coefficients of~$p_n(\gamma)$ is necessarily compact.
    Therefore, there is a sequence~$(n_k)_{k=1}^{\infty}$ such that the
    polynomials~$p_{n_k}(\gamma)$ converge coefficientwise to
    \[
    p_{*}(\gamma) =(-1)^{\frac{m(m+1)}2}\gamma^{m+1}+\dots.
    \]
    In particular,~$p_{*}(\gamma)\not\equiv 0$ due to the constant leading coefficient. The last
    condition however cannot be satisfied. Indeed, Lemma~\ref{lm:strong2a} yields
    that~\( \ind_0(\mt)=\ind_0(\mh)=0\). By Hamburger's
    criterion~\cite[pp.~183--185]{HamburgerT3}, this is equivalent to
    \[%begin{equation}\label{eq:limit_det}
        \lim_{n\to\infty}\frac{\Delta_{n}^{(0)}[\mt]}{\Delta_{n-1}^{(1)}[\mt]} = 0
        \quad\text{and}\quad
        \lim_{n\to\infty}\frac{\Delta_{n-k}^{(k)}[\mt]}{\Delta_{n-k-1}^{(k+1)}[\mt]}
        = \xi_k(\gamma),
        \quad%\text{where}\quad
        0<\xi_k(\gamma)<\infty,
        \quad
        k=1,2,\dots,
    \]%end{equation}
    for some functions~$\xi_k(\gamma)$ and all~$\gamma\in(-\varepsilon,\varepsilon)$. The ratios
    are well defined, because the inequality~$\Delta_{n-k}^{(k)}[\mt]>0$ follows from positive
    definiteness of the quadratic form~\eqref{eq:quadr_form_gamma}. Consequently,
    \[
        0=\lim_{n\to\infty}
        \frac{\Delta_{n}^{(0)}[\mt]}{\Delta_{n-1}^{(1)}[\mt]}
        =
        \lim_{n\to\infty}
        \frac{\det F_{n-m-1}^{(m+1)}[\mt]}{\left(\Delta_{n-m-1}^{(m+1)}[\mt]\right)^{m}\Delta_{n-1}^{(1)}[\mt]}
        =
        \frac{1}{\prod_{k=1}^{m}\xi_k(\gamma)}\lim_{n\to\infty}
        \frac{\det F_{n-m-1}^{(m+1)}[\mt]}{\left(\Delta_{n-m-1}^{(m+1)}[\mt]\right)^{m+1}}
        ,
    \]
    that is~$p_{*}(\gamma)= 0$ for all~$\gamma\in(-\varepsilon,\varepsilon)$ and,
    thus,~$p_{*}(\gamma)\equiv 0$. This contradiction proves the lemma.
\end{proof}

\begin{proof}[Proof of Theorem~\ref{th:strong2}]
    %\xxtextls{
    The theorem's hypothesis for the Hamburger sequences reads: there exist some~$\varepsilon>0$
    and~$m\ge 0$ such that
    \[
    (c_0,\dots,c_{m-1},c_{m}^*,c_{m+1}\dots)
    \]
    is a Hamburger moment sequence for all~$c_{m}^*$ satisfying~$|c_{m}-c_{m}^*|<\varepsilon$.
    The case~$m=0$ follows immediately from Hamburger's criterion of determinacy, so we
    assume~$m>0$.

    Let~$\mh=(c_0,c_1,c_2,\dots)$ be determinate, then Lemma~\ref{lm:strong2a} implies
    that~$r\coloneqq \ind_0(\mh)<\frac m2$. Therefore, the trimmed sequence
    \[
    (c_{2r},\dots,c_{m-1},c_{m},c_{m+1},\dots)
    \]
    is determinate with the zero index by definition of the index. The last assertion however
    contradicts Lemma~\ref{lm:strong2b}, so~\mh\ cannot be determinate.

    Now, let~$\varepsilon>0$ and let~$(s_0,\dots,s_{m-1},s_{m}^*,s_{m+1}\dots)$ be a Stieltjes
    moment sequence for all~$s_{m}^*$ satisfying~$|s_{m}-s_{m}^*|<\varepsilon$.
    Then
    \[
    (s_0,0,\dots,s_{m-1},0,s_{m}^*,0,s_{m+1},0\dots)
    \]
    is a symmetric Hamburger moment sequence; it is indeterminate by the first part of
    Theorem~\ref{th:strong2} which is already proved. As a result, the correspondence between
    Stieltjes and symmetric moment sequences yields
    \[
    (s_0,\dots,s_{m-1},s_{m},s_{m+1}\dots)\in\IDetm_S
    .
    \]
\end{proof}

\end{document}